\documentclass[12pt]{article}
\usepackage{amsmath,amssymb,amsthm,comment}

\newtheorem*{theorem*}{Theorem}
\newtheorem{theorem}{Theorem}[section]
\newtheorem{lemma}[theorem]{Lemma}

\def\barr{\begin{array}}
\def\earr{\end{array}}

\title{On the number of edges of cyclic subgroup graphs of finite groups}
\author{Marius T\u arn\u auceanu}
\date{January 31, 2023}

\begin{document}

\maketitle

\begin{abstract}
In this note, we show that among finite nilpotent groups of a given order or finite groups of a given odd order, the cyclic group of that order has the minimum number of edges in its cyclic subgroup graph. We also conjecture that this holds for arbitrary finite groups.
\end{abstract}

{\small
\noindent
{\bf MSC2000\,:} Primary 20D60; Secondary 20D15, 05C07.

\noindent
{\bf Key words\,:} cyclic subgroup graphs, finite groups, element orders.}

\section{Introduction}

Recently, mathematicians constructed many graphs which are assigned to groups by different methods (see e.g. \cite{4}). One of them is the \textit{subgroup graph} $L(G)^*$ of a group $G$, that is the graph whose vertices are the subgroups of $G$ and two vertices, $H_1$ and $H_2$, are connected by an edge if and only if $H_1\leq H_2$ and there is no subgroup $K\leq G$ such that $H_1<K<H_2$. Finite groups with planar subgroup graph have been classified by Starr and Turner \cite{14}, Bohanon and Reid \cite{3}, and Schmidt \cite{13}. Also, the genus of $L(G)^*$ has been studied by Lucchini \cite{10}.

In the current note, we will focus on a remarkable subgraph of $L(G)^*$, namely the \textit{cyclic subgroup graph} $C(G)^*$ of $G$. Its vertex set is the poset $C(G)$ of cyclic subgroups of $G$ and, similarly, two vertices $H_1$ and $H_2$ are connected by an edge if and only if $H_1\leq H_2$ and $H_1<K<H_2$ for no cyclic subgroup $K$ of $G$. Note that the (cyclic) subgroup graph of a group $G$ is the Hasse diagram of the poset of (cyclic) subgroups of $G$, viewed as a simple, undirected graph.\newpage

We will prove that if $G$ is a nilpotent group of order $n$ or a group of odd order $n$, then the minimum number of edges of $C(G)^*$ is attained for the cyclic group $\mathbb{Z}_n$. This fact was somewhat expected, since finite group theory abounds with functions attaining their minimum/maximum at cyclic groups, such as the function sum of element orders \cite{1}, the function number of (cyclic) subgroups \cite{12,6,7} or the function number of edges of (directed) power graph \cite{5}.

Our main result is as follows.

\begin{theorem}
Let $G$ be a finite group of order $n$. If $G$ is nilpotent or $n$ is odd, then
\begin{equation}
|E(C(G)^*)|\geq|E(C(\mathbb{Z}_n)^*)|.\nonumber
\end{equation}Moreover, we have equality if and only if $G\cong\mathbb{Z}_n$.
\end{theorem}

For the proof of Theorem 1.1, we need the following theorem.

\begin{theorem*}
Let $G$ be a finite group of order $n$. Then there is a bijection $f:G\longrightarrow\mathbb{Z}_n$ such that $o(a)$ divides $o(f(a))$, for all $a\in G$.
\end{theorem*}

This has been formulated as a question by Isaacs (see Problem 18.1 in \cite{11}) and proved for solvable groups by Ladisch \cite{9}. A proof for arbitrary groups has been recently proposed by Amiri \cite{2}. We point out that we need this bijection only for groups of odd order, which are solvable.

We conjecture that the inequality $|E(C(G)^*)|\geq|E(C(\mathbb{Z}_n)^*)|$ also holds for any group $G$ of even order $n$. Note that in this case there are examples of non-cyclic groups $G$ of order $n$ with $|E(C(G)^*)|=|E(C(\mathbb{Z}_n)^*)|$, such as
\begin{equation}
|E(C(A_4)^*)|=|E(C({\rm Dic}_3)^*)|=|E(C(\mathbb{Z}_{12})^*|=7.\nonumber
\end{equation}

Most of our notation is standard and will usually not be repeated here. For basic notions and results on groups we refer the reader to \cite{8}.

\section{Proof of the main result}

We start by proving some auxiliary results.

\begin{lemma}
Let $G_i$, $i=1,...,k$, be finite groups of coprime orders and $G=G_1\times\cdots\times G_k$. Then
\begin{equation}
|E(C(G)^*)|=\left(\sum_{i=1}^k\frac{|E(C(G_i)^*)|}{|C(G_i)|}\right)\prod_{i=1}^k|C(G_i)|.
\end{equation}In particular, if $n=p_1^{n_1}\cdots p_k^{n_k}$ is the decomposition of $n\in\mathbb{N}^*$ as a product of prime factors, then
\begin{equation}
|E(C(\mathbb{Z}_n)^*)|=\left(\sum_{i=1}^k\frac{n_i}{n_i+1}\right)\prod_{i=1}^k(n_i+1).
\end{equation}
\end{lemma}

\begin{proof}
The equality (1) follows immediately by induction on $k$. For the equality (2), it suffices to observe that $|C(\mathbb{Z}_{p_i^{n_i}})|=n_i+1$ and $|E(C(\mathbb{Z}_{p_i^{n_i}})^*)|=n_i$, for all $i=1,...,k$.
\end{proof}

A simple way to count the number of edges in the cyclic subgroup graph of a finite group is given by the following lemma.

\begin{lemma}
Let $G$ be a finite group. Then
\begin{equation}
|E(C(G)^*)|=\sum_{a\in G}\frac{\omega(o(a))}{\varphi(o(a))}\,,
\end{equation}where $\omega(d)$ is the number of distinct primes dividing $d\in\mathbb{N}^*$ and $\varphi$ is the Euler's totient function.
\end{lemma}

\begin{proof}
We have
\begin{equation}
|E(C(G)^*)|=\sum_{H\in C(G)}|{\rm Max}(H)|,\nonumber
\end{equation}where ${\rm Max}(H)$ denotes the set of maximal subgroups of $H\in C(G)$. Clearly, this leads to (3) because a cyclic subgroup $H=\langle a\rangle$ of $G$ has $\omega(o(a))$ maximal subgroups and $\varphi(o(a))$ generators.
\end{proof}

The next lemma shows that the function in the right side of (3) is decreasing with respect to divisibility on the set of odd positive integers.

\begin{lemma}
Let $d$ and $d'$ be two odd positive integers such that $d\mid d'$ and $d\geq 3$. Then
\begin{equation}
\frac{\omega(d)}{\varphi(d)}\geq\frac{\omega(d')}{\varphi(d')}\,.
\end{equation}Moreover, we have equality if and only if either $d=d'$ or $d$ is a prime power $p^{\alpha}$ with $p\geq 5$ and $d'=3d$.
\end{lemma}

\begin{proof}
Let $d=p_1^{\alpha_1}\cdots p_r^{\alpha_r}$ and $d'=p_1^{\beta_1}\cdots p_r^{\beta_r}p_{r+1}^{\beta_{r+1}}\cdots p_s^{\beta_s}$, where $1\leq\alpha_i\leq\beta_i$ for all $i=1,...,r$ and $r\leq s$. Then (4) becomes
\begin{equation}
p_1^{\beta_1-\alpha_1}\cdots p_r^{\beta_r-\alpha_r}p_{r+1}^{\beta_{r+1}-1}\cdots p_s^{\beta_s-1}(p_{r+1}-1)\cdots(p_s-1)\geq\frac{s}{r}\,,\nonumber
\end{equation}which is true for $s=r$. So, we can assume that $s\geq r+1$. We remark that it suffices to show
\begin{equation}
(p_{r+1}-1)\cdots(p_s-1)\geq\frac{s}{r}\,.\nonumber
\end{equation}Since each $p_i$ is odd, we get
\begin{equation}
(p_{r+1}-1)\cdots(p_s-1)\geq 2^{s-r}.\nonumber
\end{equation}On the other hand, we can easily see that the function $f(x)=2^{x-r}-\frac{x}{r}$ is increasing for $x\geq r+1$. Thus
\begin{equation}
2^{s-r}-\frac{s}{r}\geq 2-\frac{r+1}{r}\geq 0,\nonumber
\end{equation}as desired.

Note that the equality occurs in (4) if and only if either
\begin{equation}
r=s \mbox{ and } \alpha_i=\beta_i, \forall\, i=1,...,r\nonumber
\end{equation}or
\begin{equation}
r=1, s=2, \alpha_1=\beta_1, \beta_2=1 \mbox{ and } p_2=3,\nonumber
\end{equation}that is if and only if either $d=d'$ or $d=p_1^{\alpha_1}$ with $p_1\geq 5$ and $d'=3d$. This completes the proof.
\end{proof}

We mention that the inequality (4) is not true when $d'$ is even. For example, for $d$ odd and $d'=2d$ we have
\begin{equation}
\frac{\omega(d)}{\varphi(d)}<\frac{\omega(d)+1}{\varphi(d)}=\frac{\omega(d')}{\varphi(d')}\,.\nonumber
\end{equation}

We are now able to prove our main result.

\bigskip\noindent{\bf Proof of Theorem 1.1.} Assume first that $G$ is nilpotent and let $G=G_1\times\cdots\times G_k$ be its decomposition as a direct product of Sylow subgroups. Denote $|G_i|=p_i^{n_i}$, $i=1,...,k$. By Theorem 2.5 of \cite{7}, we have
\begin{equation}
|C(G_i)|\geq |C(\mathbb{Z}_{p_i^{n_i}})|=n_i+1,\nonumber
\end{equation}with equality if and only if $G_i\cong\mathbb{Z}_{p_i^{n_i}}$. Using Lemma 2.2, this leads to
\begin{align*}
|E(C(G_i)^*)|&=\sum_{H\in C(G_i)}|{\rm Max}(H)|=\sum_{H\in C(G_i)\setminus 1}1=|C(G_i)|-1\\
&\geq |C(\mathbb{Z}_{p_i^{n_i}})|-1=|E(C(\mathbb{Z}_{p_i^{n_i}})^*)|=n_i.\nonumber
\end{align*}Then equalities (1) and (2) in Lemma 2.1 imply that
\begin{equation}
|E(C(G)^*)|\geq|E(C(\mathbb{Z}_n)^*)|,\nonumber
\end{equation}with equality if and only if $G\cong\mathbb{Z}_n$.

Assume next that $n$ is odd and let $f:G\longrightarrow\mathbb{Z}_n$ be a bijection such that $o(a)$ divides $o(f(a))$, for all $a\in G$. Then
\begin{equation}
\frac{\omega(o(a))}{\varphi(o(a))}\geq\frac{\omega(o(f(a)))}{\varphi(o(f(a)))}\,, \forall\, a\in G\nonumber
\end{equation}by Lemma 2.3 and therefore we have
\begin{equation}
|E(C(G)^*)|=\sum_{a\in G}\frac{\omega(o(a))}{\varphi(o(a))}\geq\sum_{a\in G}\frac{\omega(o(f(a)))}{\varphi(o(f(a)))}=|E(C(\mathbb{Z}_n)^*)|.\nonumber
\end{equation}

Finally, assume that $|E(C(G)^*)|=|E(C(\mathbb{Z}_n)^*)|$, but $G$ is not cyclic. Then
\begin{equation}
\frac{\omega(o(a))}{\varphi(o(a))}=\frac{\omega(o(f(a)))}{\varphi(o(f(a)))}\,, \forall\, a\in G.\nonumber
\end{equation}Let $a_1, ..., a_{\varphi(n)}\in G$ with $o(f(a_i))=n$ for all $i$. By Lemma 2.3, it follows that there is a prime $p$ and a positive integer $\alpha$ such that \begin{equation}
n=3p^{\alpha} \mbox{ and } o(a_i)=p^{\alpha}, \forall\, i=1,...,\varphi(n).\nonumber
\end{equation}Then, from Sylow's theorems, we infer that $G$ has a unique (normal) Sylow $p$-subgroup. On the other hand, since $G$ contains at least $\varphi(n)=2\varphi(p^{\alpha})$ elements of order $p^{\alpha}$, it will contain at least two cyclic subgroups of order $p^{\alpha}$, a contradiction.

The proof of Theorem 1.1 is now complete.$\qed$
\bigskip

\noindent{\bf Acknowledgments.} The author is grateful to the reviewer for remarks which improve the previous version of the paper.

\vspace*{3ex}\small

\hfill
\begin{minipage}[t]{5cm}
Marius T\u arn\u auceanu \\
Faculty of  Mathematics \\
``Al.I. Cuza'' University \\
Ia\c si, Romania \\
e-mail: {\tt tarnauc@uaic.ro}
\end{minipage}

\end{document}